\documentclass[11pt]{article}
\usepackage{latexsym}
\usepackage{graphicx}
\usepackage{latexsym}
\usepackage{amsfonts,amsmath,amssymb}
\usepackage{url}

\newcommand{\bpr}{\begin{prop}}
\newcommand{\epr}{\end{prop}}
\newcommand{\beq}{\begin{equation}}

\newcommand{\bona}{\mathrm{BoNa}}

\newcommand{\av}{\operatorname{Av}}

\newcommand{\qed}{\mbox{$\Diamond$}\vspace{\baselineskip}}

\newtheorem{theorem}{Theorem}[section]
\newtheorem{proposition}[theorem]{Proposition}
\newtheorem{lemma}[theorem]{Lemma}
\newtheorem{definition}[theorem]{Definition}
\newtheorem{corollary}[theorem]{Corollary}
\newtheorem{example}[theorem]{Example}

\newenvironment{proof}{\noindent {\bf Proof:}}{{\qed}}

\newcommand{\vanish}[1]{} 
\begin{document}
\title{Boolean-Narayana numbers}

\author{Mikl\'os B\'ona
\thanks{University of Florida, Gainesville FL 32611-8105.
Email: bona@ufl.edu.}
}

\date{}

\maketitle

\begin{abstract}  We introduce a refinement of Boolean-Catalan numbers and call them Boolean-Narayana numbers. We provide an explicit formula for these numbers,
and prove unimodality, log-concavity, and real-roots-only results for their sequences. We also prove a three-term recurrence relation for their generating polynomials.
\end{abstract}

\section{Introduction}
A {\em binary plane  tree} is a rooted plane tree on unlabeled vertices in which every non-leaf vertex has one or two children, and every child is a left child or a right child, even if it is the only child of its parent. It is well-known that the number of such trees on $n$ vertices is the $n$th Catalan number $C_n={2n\choose n}/(n+1)$. 
We will consider an enhanced version
of these trees. 
\begin{definition} A {\em 0-1-tree} is a binary plane tree in which every vertex that has two children is labeled 0 or 1. \end{definition}

Figure \ref{trees-0-1} shows the six 0-1-trees on three vertices. 
\begin{figure}[htb] \label{trees-0-1}
\begin{center}\includegraphics[width=6cm]{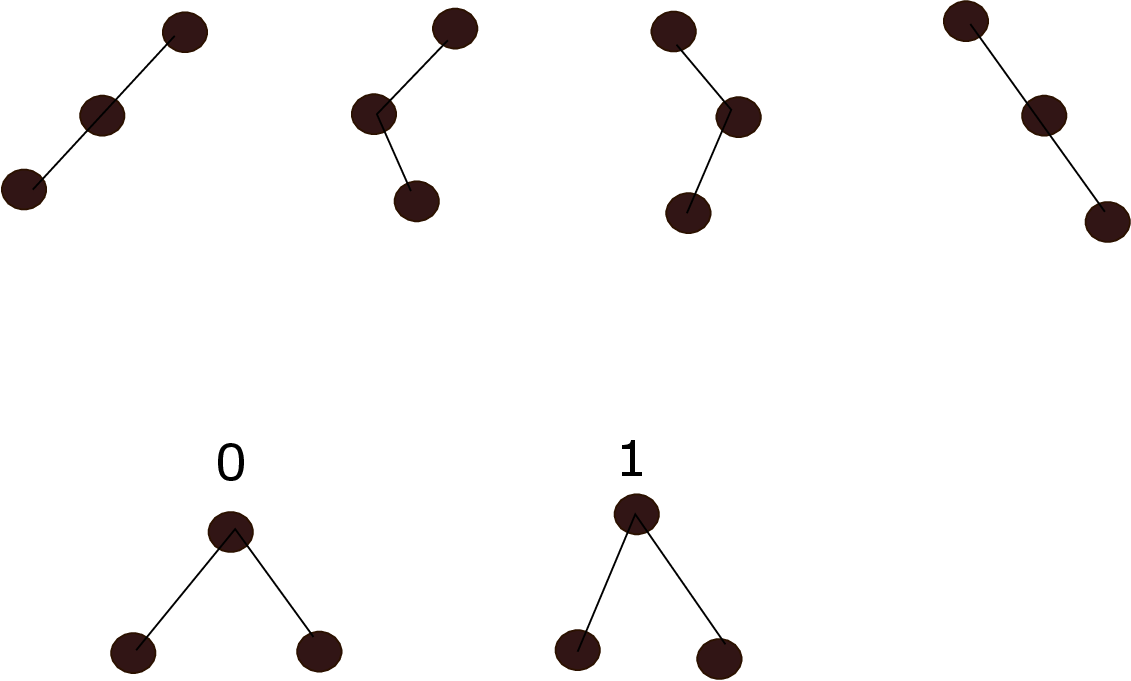}
\caption{The six 0-1 trees on three vertices.}
\end{center} \end{figure}

The number of 0-1-trees on $n$ vertices is denoted by $B(n)$ and is called a {\em Boolean-Catalan number} \cite{hossain}. The Boolean-Catalan numbers count various objects, some of which
are discussed in the thesis of Chetak Hossain \cite{hossain}, and some others appear in papers of Colin Defant \cite{aust,lothar,enum}, later to be shown to bijectively
correspond to 0-1-trees by the present author \cite{01trees}. The name "Boolean-Catalan numbers" can be justified, for example, by saying that binary plane trees are
counted by the Catalan numbers, and if a binary plane tree has $m$ vertices with two children, then we associate an element of the Boolean algebra $\{0,1\}^m$ to it.

In this paper, we take this concept further.  The {\em Narayana numbers} are the best-known refinement of the Catalan numbers. Accordingly, they have many equivalent
definitions. The one that is closest to our subject defines the Narayana number $N(n,k)$ as the number of binary plane trees on $n$ vertices with $k-1$ right edges. 
So we  define the {\em Boolean-Narayana numbers} $\bona(n,k)$ as the number of 0-1-trees on $n$ vertices that have $k-1$ right edges. 

In Section \ref{combinatorial}, we provide two additional combinatorial interpretations for the Boolean-Narayana numbers. In Section \ref{enumeration}, we use the Lagrange
Inversion Formula to prove an explicit formula for these numbers.  Then in Section \ref{unimodal-hor}, we construct an injection to prove that  for fixed $n$, the sequence of the numbers
$\bona(n,k)$, where $k=1,2,\cdots ,n$, is unimodal. We follow this by proving in Section \ref{realzeros} that the generating polynomial $\bona_n(u)$ of that sequence has real roots only, 
which implies the sequence of its coefficients is log-concave. We also prove that the polynomials $\bona_n(u)$ have weakly interlacing roots.  
 In Section \ref{rec}, we prove a three-term recurrence relation satisfied by
the polynomials $\bona_n(u)$. Then in Section \ref{Sturm} we use this recurrence relation to provide an alternative proof
of that interlacing property of these polynomials. 
 In Section \ref{logc-vert}, we fix $k$, and prove that the infinite sequence $\bona(n,k)$, where $n=k+1,k+2,\cdots $, is log-concave. 

That is, we will see that the Boolean-Narayana numbers share the best-known properties of the Narayana numbers. This is interesting, since the explicit formula we 
obtain for these numbers is not as compact as the well-known formula $N(n,k)=\frac{1}{n}{n\choose k}{n\choose k+1}$ for the latter.
We end by discussing future directions in Section \ref{further}, where we  also explain our original motivation for this family of questions. 

\section{A combinatorial interpretation}  \label{combinatorial}
{\em Stack sorting} or {\em West stack sorting} of permutations is an interesting and well-studied field of research. See Chapter 8 of \cite{combperm} for an introduction. 
The stack sorting operation $s$ has three natural,  equivalent definitions. Here we mention one of them. We define $s$ on all finite permutations $p$ {\em and their substrings} by setting $s(1)=1$ for the one-element permutation, and setting $s(p)=s(LnR)=s(L)s(R)n$, where $n$ is the maximum entry of the string $p$, while $L$ is the substring of entries on the left of $n$, and $R$ is the substring on the entries on the right of $n$. Note that $L$ and $R$ can be empty. 

The questions, methods and results about the stack sorting algorithm are intrinsically connected to the theory of pattern-avoiding permutations, a large and rapidly evolving field. 
 We say that a permutation $p$ {\em contains} the pattern (or subsequence) $q=q_1q_2\cdots q_k$ 
if there is a $k$-element set of indices $i_1<i_2< \cdots <i_k$ such that $p_{i_r} < p_{i_s} $ if and only
if $q_r<q_s$.  If $p$ does not contain $q$, then we say that $p$ {\em avoids} $q$. For example, $p=3752416$ contains
$q=2413$, as the first, second, fourth, and seventh entries of $p$ form the subsequence 3726, which is order-isomorphic
to $q=2413$.  A recent survey on permutation 
patterns by Vatter can be found in \cite{vatter}.

Recently, in a series of papers \cite{aust,lothar,enum} Colin Defant studied the size of $s^{-1}(Class_n)$, that is, he studied the number of permutations $p$ of length $n$
for which $s(p)$ belongs to a certain permutation class.  Let $\av_n(q)$ denote the set of permutations of length $n$ that avoid the pattern $q$. The classic result that the number
of permutations $p$ of length $n$ for which $s(p)$ is the identity permutation is the Catalan number $C_n$  is a special case of this line of work.  Namely, it is the result that
$|s^{-1}(\av_n(21))|=C_n$. Considering various permutation classes defined by sets of patterns of length three,  Defant has proved the chain
of equalities

\begin{eqnarray*}\sum_{n\geq 1} |s^{-1}(\av_n(132,312))| z^n & = &  \sum_{n\geq 1} |s^{-1}(\av_n(231,312))| z^n \\ = \sum_{n\geq 1} |s^{-1}(\av_n(132,231))|z^n 
& = & \frac{1-2z-\sqrt{1-4z-4z^2}}{4z} .\end{eqnarray*}

In \cite{01trees}, the present author showed the equality $\sum_{n\geq 1}B(n)z^n =  \frac{1-2z-\sqrt{1-4z-4z^2}}{4z}$, and then
he constructed a bijection (of a recursive nature) from the set   $s^{-1}(\av_n(231,312))$ of permutations to the set $\mathcal B_n$ of 0-1-trees on $n$ vertices, and another bijection from the 
set $s^{-1}(\av_n(132,312))$ of permutations to $\mathcal B_n$. In this section, we will use the ideas presented there to establish combinatorial interpretations of the
Boolean-Narayana numbers.  If $p=p_1p_2\cdots p_n$ is a permutation, then we say that $i$ is an {\em ascent} of $p$ if $p_i<p_{i+1}$. If $p_i > p_{i+1}$, then we say that 
$i$ is a {\em descent} of $p$. 

\begin{proposition} \label{first} The number of permutations in $s^{-1}(\av_n(231,312))$ that have exactly $k-1$ descents is $\bona(n,k)$. \end{proposition}
\begin{proof} Let $p=LnR$ be a permutation of length $n$. For $s(p)$ to avoid both 231 and 312, it is necessary that both $L$ and $R$ avoid those two patterns. 
If $L$ and $R$ are both nonempty, then it is proved in \cite{01trees} that regardless of what $L$ and $R$ are, there are exactly two possible choices for the underlying set
of $L$ (the underlying set of $R$ is determined by these choices) so that $s(p)\in \av_n(231,312)$. If $L$ (resp. $R$) is empty, then the sufficient and necessary condition for 
$s(p)\in \av_n(231,312)$ is that $s(R)$ (resp $s(L)$) avoids both 231 and 312.  

Let $A(n,k)$ be the number of permutations in $s^{-1}(\av_n(231,312))$ that have exactly $k-1$ descents. Then $A(1,1)=1$, and it follows from the analysis in the previous paragraph that
\begin{equation} \label{recurrence} A(n,k) = A(n-1,k)+A(n-1,k-1) +2 \sum_{i=2}^{n-1} \sum_{j\geq 0} A(i-1,j)A(n-i,k-j) .\end{equation}
Indeed, on the right-hand side, the first summand counts permutations in which $n$ is in the last position, the second summand counts those in which $n$ is in the 
first position, and the third summand counts those where $n$ is in the middle and counts them according to the number $j-1$ of descents in $L$. 

On the other hand, $\bona(1,1)=1$, and  the numbers $\bona(n,k)$ obviously satisfy the analogous recurrence relation 
\begin{equation} \label{rbecurrence} \bona(n,k) = \bona(n-1,k)+\bona(n-1,k-1) +2 \sum_{i=2}^{n-1} \sum_{j\geq 0} \bona(i-1,j)\bona(n-i,k-j) ,\end{equation}
since the number of right edges in a 0-1 tree is easy to count by considering subtrees of the root.
This proves our claim.
\end{proof}

\begin{proposition} \label{second} The number of permutations in $s^{-1}(\av_n(132,312))$ that have exactly $k-1$ descents is $\bona(n,k)$. \end{proposition}
\begin{proof} Analogous to the proof of Proposition \ref{first}, with the needed statement on the underlying sets of $L$ and $R$ proved in \cite{01trees}.
\end{proof}

Note that  it follows by reflection through a vertical axis that $\bona(n,k)=\bona(n,n-k+1)$. So Propositions \ref{first} and \ref{second} yield the following.

\begin{corollary}  \label{symmetry} For all $n$ and $k$, the equality $A(n,k)=A(n,n-k+1)$ holds. That is, 
the number of permutations in  $s^{-1}(\av_n(231,312))$ with $k-1$ ascents is equal to the number of permutations in  $s^{-1}(\av_n(231,312))$ with 
$k-1$ descents. The same holds for $s^{-1}(\av_n(132,312))$.
\end{corollary}

\section{Enumeration}  \label{enumeration}
In this section, we will use the Lagrange Inversion Formula to find an explicit formula for $\bona(n,k)$.  The Lagrange Inversion Formula, in its simplest form, 
can be found in many graduate-level textbooks. However, in this section, we will use its univariate version on a function that has two variables, and for that reason, 
we will explain which version we use and how. 

The Lagrange Inversion Formula, in one of its univariate versions, is the following. 
\begin{theorem} \label{lif}  Let 
$F(z)$ be the unique formal power series with
$F(0)=0$ satisfying 
\begin{equation} \label{lif-def} F(z)=z\phi(F(z),\end{equation}
where $\phi(y)$ is a formal power series with $\phi(0)\neq 0$. 
Then for all positive integers $n$, the equality
\[[z^n]F(z)=\frac{1}{n}[y^{n-1}\phi(y)^n] \] holds.
\end{theorem}

Note that the name "inversion formula" is justified because (\ref{lif-def}) is equivalent to the fact that $F$, as a function of $z$,  is the {\em compositional inverse} of
the function $z/ \phi(z)$. 

Now let us assume that we have a {\em bivariate} power series $F(z,u)$ that is defined implicitly by the equation
\begin{equation} \label{bivariate} F(z,u)=z\phi(F(z,u),u), \end{equation} where $\phi(0,u)\neq 0$.  
Then we can treat $u$ as a parameter and apply the univariate formula (\ref{lif-def}) in the variable $z$, since (\ref{bivariate}) is equivalent to stating that the compositional 
inverse of $F(z,u)$, in the variable $z$, is equal to $z/\phi(F(z,u),u)$.

 As $\bona(n,k)=\bona(n,n-k+1)$,
it suffices to compute $\bona(n,k)$ for the case when $k\leq (n+1)/2$.

\begin{theorem}  \label{explicit}
Let $k\leq (n+1)/2$, where $n$ and $k$ are positive integers. Then
\[\bona(n,k)=\frac{1}{n} {n\choose k-1} \sum_{j=0}^{k-1} 2^j {k-1\choose j}{n-k+1\choose j+1} .\]
\end{theorem}

\begin{proof} Let $t(n,k)=\bona(n,k+1)$ be the number of 0-1 trees on $n$ vertices with $k$ right edges, and let
\begin{equation} \label{bivariate-tree} T(z,u)=\sum_{n\geq 1,k\geq 0} t(n,k) z^n u^k = \sum_{n\geq 1,k\geq 1} \bona(n,k)z^n u^{k-1}. \end{equation} 
Removing the root of a 0-1 tree we either get the empty set, or one tree that was connected to the root by a left edge, or one tree that was connected to the
root by a right edge, or two trees, one of which was connected to the root by a left edge.  This implies the functional equation
\[T=z(1+uT+T+2uT^2).\]
In order to use the Lagrange Inversion Formula, we set 
\begin{equation} \label{defofphi} \phi(y)=\phi(y,u)=1+(u+1)y+2uy^2,\end{equation} and then (\ref{defofphi}) becomes
\begin{equation} \label{lagrange-first} T=z\phi(T)) .\end{equation} 

By Theorem \ref{lif}, 
\begin{equation} \label{lagrange-second} [z^nu^{k-1}] T(z,u)=\frac{1}{n}[y^{n-1}u^{k-1}] \phi(y)^n.\end{equation}

Now notice that 
\[\phi(y)=(1+y)+u(y+2y^2)=(1+y)+uy(1+2y) .\]
Therefore, for $\phi(y)^n$ on the right-hand side of (\ref{lagrange-second}), we take the power $((1+y)+uy(1+2y))^n$, and notice
that the term $u^{k-1}$ will come from the summand of that binomial expansion in which the $u$-term is taken to the $(k-1)$st power. 
That is, 
\[[u^{k-1}]\phi(y)^n = {n\choose k-1} (1+y)^{n-k+1}y(1+2y))^{k-1} .\] 
Recalling (\ref{lagrange-second}), this means
\begin{eqnarray*} \bona(n,k) & = & [z^nu^{k-1}]T(z,u)=\frac{1}{n} {n\choose k-1}[y^{n-1} ]y^{k-1}\cdot(1+2y)^{k-1} (1+y)^{n-k+1} \\
& = & [z^nu^{k-1}]T(z,u)=\frac{1}{n} {n\choose k-1}[y^{n-k} ](1+2y)^{k-1} (1+y)^{n-k+1}.\end{eqnarray*}
Finally, expanding the binomial sums we get the expression
\[\bona(n,k)=\frac{1}{n} {n\choose k-1} \sum_{j=0}^{\min(k-1,n-k)} 2^j {k-1\choose j}{n-k+1\choose j+1} ,\]
However, recall that we assumed that $k\leq (n+1)/2$, (which, for symmetry reasons, does not result in loss of generality), so 
$k-1\leq n-k$, and our theorem is proved.  
\end{proof}

The numbers $\bona(n,k)$ for $1\leq n\leq 5$, and $k=1,2,\cdots ,n$ are shown below. Note the interesting fact that $\bona(n,2)=(n-1)^2$.

\begin{itemize}
\item 1
\item 1, 1
\item 1, 4, 1
\item 1, 9, 9, 1
\item 1, 16, 38, 16, 1.
\end{itemize}

\section{A Combinatorial Proof for Unimodality for fixed $n$} \label{unimodal-hor}
We say that sequence $a_1,\cdots ,a_n$ is {\em unimodal} if there exists an index $k$ so that $a_1\leq a_2\leq \cdots a_k\geq a_{k+1} \geq \cdots \geq a_n$.
In this section we present a combinatorial proof of the following. 

\begin{theorem} \label{unimodaltheorem}
For all fixed $n$, the sequence $\bona(n,k)_{1\leq k \leq n}$ is
unimodal. 
\end{theorem}

\begin{proof} As for any fixed $n$, the sequence $\bona(n,k)_{1\leq k \leq n}$ is symmetric, it suffices to show that $\bona(n,k)\leq \bona(n,k+1)$ if $k\leq (n-1)/2$. 
Let $\mathcal BoNa(n,k)$ be the set of 0-1 trees with $k-1$ right edges. We will construct an injection $z: \mathcal BoNa(n,k)\rightarrow \mathcal BoNa(n,k+1)$.
Note that our injection $z$ is a slightly modified version that the present author used to prove a unimodality result on $t$-stack sortable permutations in
\cite{stacksort, stacksort1}. A perhaps more accessible source is Section 8.2.3 of \cite{combperm}. 

In order to define this injection $z$, we need two tools. First, we define an involution $f:\mathcal B(n)\rightarrow \mathcal B(n)$, where $\mathcal B(n)$ is the set
of all 0-1 trees on $n$ vertices, as follows. Let $T\in \mathcal B(n)$, and let $v$ be a vertex of $T$.
Construct $f(T)$ as follows.
\begin{enumerate}
\item[(a)] If  $v$ has zero or two children, leave the subtrees of $v$
unchanged.
\item[(b)] If $v$ has a left subtree only, then turn that subtree into a
right subtree. 
\item[(c)]  If $v$ has a right subtree only, then turn that subtree into a
left subtree. 
\end{enumerate}
Leave all the labels unchanged. 
Let $f(T)$ be the tree we obtain from $T$ in this way. See Figure 2 for an example.

\begin{figure}[htb] \label{fig:involution}
\begin{center}\includegraphics[width=7cm]{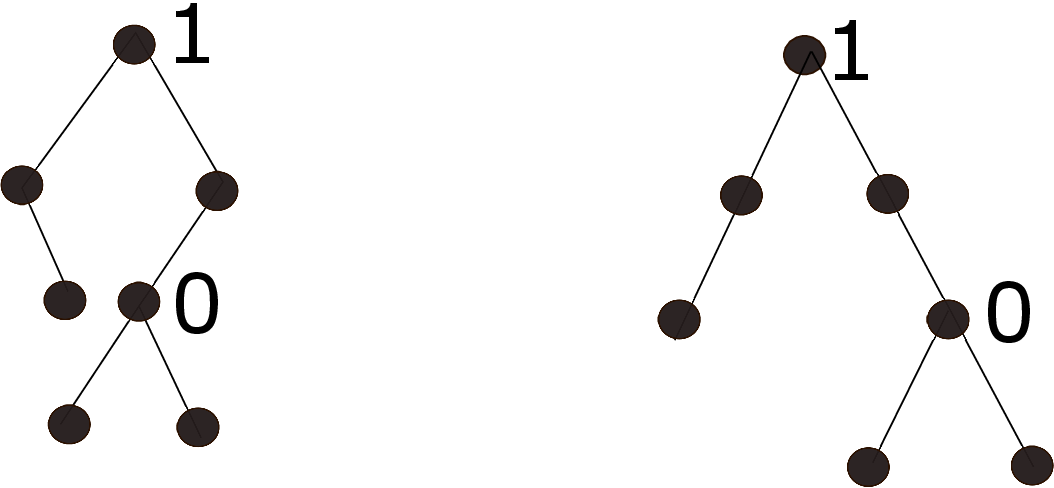}
\caption{A 0-1 tree $T$ and its image $f(T)$.}
\end{center} \end{figure}

Note that if $T$ has $k-1$ right edges, then $f(T)$ has $k-1$ left edges (and so, $n-k$ right edges). 

Second, for each $T\in \mathcal B(n)$, we define a total
order of the nodes of $T$ as follows. Let us say that a
node $v$ of $T$ is on {\em level} $j$ of $T$ if the distance of $v$ from
the root of $T$ is $j$. Then our total order consists of listing  the
nodes on the highest level of $T$ going from left to right, then the
nodes on the second highest level left to right, and so on, ending with
the root of $T$. 

Let $T_i$ be the subgraph of $T$ induced by the smallest $i$ vertices in
this total order. Then $T_i$ is either a tree or a forest with at least
two components. If $T_i$ is a
tree, then $f(T_i)$ is a tree as given by the definition of the map  
$f:\mathcal B(n)\rightarrow \mathcal B(n)$.
 If $T_i$ is a forest, then we define $f(T_i)$ as the plane forest whose
$h$th component is the image of the $h$th component of $T_i$. 

Now let $k\leq  (n-1)/2  $, and let $T\in \mathcal BoNa(n,k)$. We
define $z(T)$ as follows. Take the sequence $T_1, T_2, \cdots , T_n. $ 
Denote by $\ell(T)$  the number of left 
 edges of the forest $T$, and by $r(T)$ the number of right edges of the
forest $T$.
Find the smallest index $i$ so that $\ell(T_i)-r(T_i)=1$.
We will now explain why such an index always exists. If $T_2$ is a left
edge, then $\ell(T_2)-r(T_2)=1-0=1$, and we are done. Otherwise,   at the
beginning we have
$\ell(T_2)-r(T_2)<1$, while at the end we have
 \[\ell(T_n)-r(T_n) = (n-1-(k-1)) - (k-1) =n-2k+1 \geq 1 ,\] because of the restriction  $k\leq (n-1)/2$. 
 So at the beginning, $\ell(T_j)-r(T_j)$ is too small,
while at the end, it is too large. On the other hand, it is obvious that
as $i$ changes from 2 to $n$, at no step could $\ell(T_j)-r(T_j)$ ``skip''
a value as it could only change by 1. Therefore, for continuity reasons,
it has to be equal to 1 at some point, and we set $i$ to be the smallest
index for which this happens. 

Now apply $f$ to the forest $T_i$, and leave the rest of $T$
unchanged. Let $z(T)$ be the obtained tree. Before we prove that $z$ is
an injection, let us consider an example. 

\begin{example}
Figure 3 shows an example of the injection $z:\mathcal BoNa(9,4)\rightarrow \mathcal BoNa(9,5)$.
The large numbers denote the label of vertices with two children, whereas the small labels show the place of a
vertex in the total order we defined above. In this example, the smallest index $i$ so that 
$\ell(T_i)-r(T_i)=1$ is $i=3$.

 \begin{figure}[htb] \label{injection}
\begin{center}\includegraphics[width=7cm]{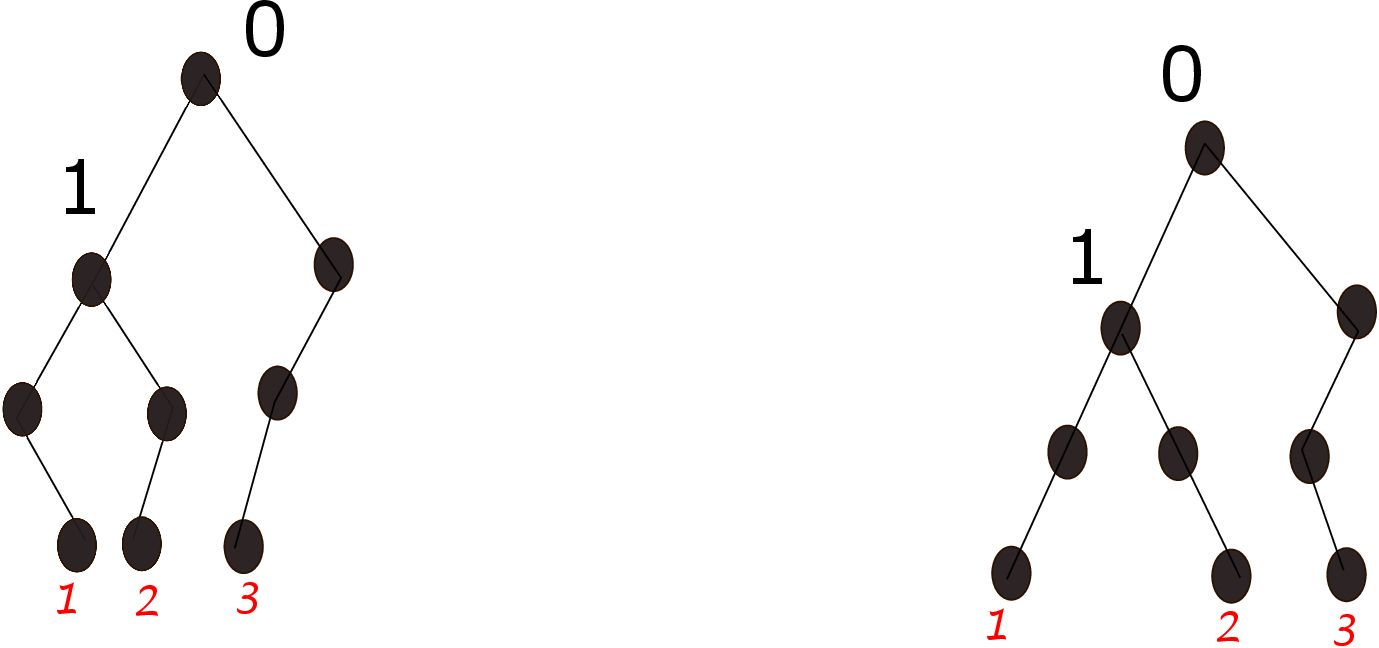}
\caption{The action of $z$ on a tree in $\mathcal BoNa(9,4)$.}
\end{center} \end{figure}
\end{example}

\begin{lemma} \label{zinj}
For all $k\leq  (n-1)/2  $, the function $z$ described
above is an injection from $\mathcal BoNa(n,k)$ into 
$\mathcal BoNa(n,k+1)$.
\end{lemma}

\begin{proof} 
It is clear that $z$ maps into $\mathcal BoNa(n,k+1)$ as $z$ consists of the
application of $f$ to a subgraph $T_i$ in which left edges outnumber
right edges by one. We know that $f$ turns the number of  left edges into the number of right edges
and vice versa, so $z$ indeed increases the number of right edges by 1. 

To see that $z$ is an injection, note that $z(T_i)$ is the smallest
subforest of $z(T)$ that consists of the first few nodes of $z(T)$ in
the total order of all nodes and in which right edges outnumber left
edges by one.

 Now let $U\in \mathcal BoNa(n,k+1)$. If $U$ does not have a subforest that consists
 of the first few nodes of $U$ in the total order of all nodes and in
 which the right edges outnumber the left edges by one, then by the
 previous paragraph, $U$ has no preimage under $z$. Otherwise, the
 unique preimage $z^{-1}(U)$ can be found by finding the smallest such
 subforest, applying $f$ to it, and leaving the rest of $U$ unchanged.
\end{proof}

Lemma \ref{zinj} shows that $\bona(n,k)\leq \bona(n,k+1)$ if $k\leq (n-1)/2$. This implies that the sequence $\{\bona(n,k)\}_{1\leq k\leq n}$ weakly increases to its midpoint. As the sequence
is symmetric, it follows that it weakly decreases after that. Therefore, the sequence is unimodal as claimed. 
\end{proof}

Note that $z$ never changes the labels of any vertex, nor does it change the ability of a vertex to receive a label, so this injection can be extended to other classes of labeled trees. 

\section{The real zeros property} \label{realzeros}
Let $\bona_n(u)=\sum_{k=1}^n \bona(n,k)u^k$. We call these polynomials the {\em Boolean-Narayana polynomials}.  
\begin{theorem} \label{realzerostheorem} For all  positive integers $n$, the Boolean-Narayana polynomial $\bona_n(u)$ has real roots only. 
\end{theorem} 

\begin{proof} 
Note that $\bona_n(u)=[z^n](uT(z,u))$. Note that for a fixed $u$, (\ref{lagrange-second}) implies, using (\ref{defofphi}), that
\[[z^n]T(z) =\frac{1}{n}[y^{n-1}]\phi(y)^n,\] while if we view $u$ as a variable, then we get
\[\sum_{k\geq 1}^n \bona(n,k) u^{k-1}=[z^n]T(z,u)=\frac{1}{n}[y^{n-1}]\left(1+(1+u)y+2uy^2\right)^n.\]
Multiplying both sides by $u$, we get
\begin{equation} \label{narpoly} \bona_n(u)= \frac{u}{n}[y^{n-1}]\left(1+(1+u)y+2uy^2\right)^n .\end{equation}

Now note that $\phi(y,u)$ is a quadratic polynomial in $y$ as we saw in (\ref{defofphi}). Let
\begin{equation} \label{defrs} \phi(y,u)=(1+(1+u)y+2uy^2=(1-r_1y)(1-r_2y),\end{equation} so that $r_1+r_2=-(1+u)$ and $r_1r_2=2u$.
Note that the discriminant of $\phi(y,u)$ as a quadratic polynomial of $y$ is
$(1+u)^2-8u=1-6u+u^2$, which is positive if $u<0$, so $r_1$ and $r_2$ are real for negative $u$. 
Then (\ref{narpoly}) becomes
\begin{equation} \label{bona-factored} \bona_n(u)=\frac{u}{n} [y^{n-1}](1-r_1y)^n(1-r_2y)^n.\end{equation}
Now let us expand and extract the coefficient of $y^{n-1}$ on the far-right side to get

\begin{equation} \label{almost-nar}
[y^{n-1}](1-r_1y)^n(1-r_2y)^n
=
(-1)^{n-1} r_2^{\,n-1}
\sum_{j=0}^{n-1}
\binom{n}{j}\binom{n}{j+1}
\left(\frac{r_1}{r_2}\right)^j.
\end{equation}

Recall that \(r_1\) and \(r_2\) are functions of \(u\), and set
\[
q(u)=\frac{r_1(u)}{r_2(u)}.
\]
Since the coefficients
\[
\binom{n}{j}\binom{n}{j+1}
\]
are constant multiples of the Narayana numbers, it is convenient to define the Narayana polynomial
by
\begin{equation}\label{narayana-poly}
N_n(q)=\frac{1}{n}\sum_{j=0}^{n-1}\binom{n}{j}\binom{n}{j+1}q^{j+1}.
\end{equation}
Then equation (\ref{almost-nar})  becomes
\begin{equation}  \label{bona-from-nar}
[y^{n-1}](1-r_1y)^n(1-r_2y)^n
=
(-1)^{n-1}r_2(u)^{\,n-1}\,n\,\frac{N_n(q(u))}{q(u)}.
\end{equation}

Comparing equation (\ref{bona-factored}) and equation (\ref{bona-from-nar}), we obtain
\[
\mathrm{BoNa}_n(u)
=
u(-1)^{n-1}r_2(u)^{\,n-1}\frac{N_n(q(u))}{q(u)}.
\]
Using the identities
\[
q(u)=\frac{r_1(u)}{r_2(u)}
\qquad\text{and}\qquad
r_1(u)r_2(u)=2u,
\]
we get
\[
\frac{u}{q(u)}
=
u\frac{r_2(u)}{r_1(u)}
=
\frac{r_2(u)^2}{2}.
\]
Therefore,
\begin{equation}   \label{bona-nar-explicit}
\mathrm{BoNa}_n(u)
=
\frac{(-1)^{n-1}}{2}\,r_2(u)^{\,n+1}N_n(q(u)).
\end{equation}

Equation (\ref{bona-nar-explicit}) shows the close connection between the polynomials \(\mathrm{BoNa}_n(u)\) and
\(N_n(q)\). It is the basis for our results in the rest of this section and the next.

It is known \cite{shen} that all roots of \(N_n(u)\) are real and simple. One of these roots is \(0\),
and the other \(n-1\) roots are negative. Other sources for these results include Theorem 5.3.1 in \cite{brenti}
and Theorem 3.2 in \cite{shen}.  Since \(r_2(u)\neq 0\) for \(u\le 0\), equation (\ref{bona-nar-explicit}) shows that the roots
of \(\mathrm{BoNa}_n(u)\) are precisely the preimages under \(q\) of the roots of \(N_n\). Therefore,
this will complete the proof of our theorem if we can show that the function \(q(u)\) is a bijection
from the set \((-\infty,0]\) to itself.

Set \(\Delta=1-6u+u^2\). As \(\Delta>0\) if \(u<0\), the square root of \(\Delta\) is defined as a real
number, and we set \(s=\sqrt{\Delta}\) for \(u<0\). Now note that
\[
q(u)=\frac{r_1(u)}{r_2(u)}=\frac{1+u-s}{1+u+s}.
\]
Therefore,
\[
q'(u)=\frac{-8(u-1)}{(u+1+s)^2s}.
\]
Now note that if \(u<0\), then \(-8(u-1)=8(1-u)>0\), and \((u+1+s)^2>0\), so \(q'(u)>0\), and so
\(q(u)\) is an increasing function on the interval \((-\infty,0)\). Finally, \(q\) extends continuously
to \(u=0\), with \(q(0)=0\), and
\[
\lim_{u\to -\infty} q(u)=-\infty.
\]
Therefore, \(q\) is a bijection from \((-\infty,0]\) onto \((-\infty,0]\).
\end{proof}

We say that  sequence $a_1, a_2, \cdots $ of nonnegative real numbers is  {\em log-concave}  if for all $i$, the inequality $a_{i-1}a_{k+1}\leq a_i^2$ holds. 
For finite sequences $a_1,a_2,\cdots ,a_n$, it is well known \cite{brenti} that if the generating polynomial $P(z)=\sum_{i=1}^n a_i z^i$ has real roots only, then the sequence
$a_1,a_2,\cdots ,a_n$ is log-concave. So Theorem \ref{realzerostheorem}  proves the following.

\begin{corollary} \label{logccor}  For all fixed $n$, the sequence $\bona(n,1),\bona(n,2),\cdots ,\bona(n,n)$ is log-concave. \end{corollary}
In other words, the Boolean-Narayana numbers are "horizontally log-concave". 
It would be interesting to give a combinatorial proof of this fact. We note that finite log-concave sequences of positive real numbers are unimodal, so it is perhaps not 
surprising that a combinatorial proof of Corollary \ref{logccor} is more difficult to find than such a proof of Theorem \ref{unimodaltheorem}.

We return to the function $q(u)$ and
 collect our observations about the function $q(u)$ in the following corollary for future use.
\begin{corollary} \label{properties}
If $u<0$, then $q(u)$ is a negative real number. In fact, the function 
\[q:(-\infty,0]\to(-\infty,0]\] is continuous and strictly increasing. As $q$ takes all real negative values,  it is a bijection.
\end{corollary}

We are now ready to prove our next theorem.

\begin{theorem} \label{interlacing}
The polynomials \(\{\mathrm{BoNa}_n(u)\}_{n\ge 2}\) have weakly interlacing roots. More precisely,
each polynomial \(\mathrm{BoNa}_n(u)\) has a root at \(0\), and the nonzero roots of
\(\mathrm{BoNa}_{n-1}(u)\) strictly interlace the nonzero roots of \(\mathrm{BoNa}_n(u)\).
\end{theorem}

\begin{proof}
It is well-known \cite{kostov} that the Narayana polynomials \(\{N_n(q)\}_{n\ge 2}\) have real simple roots,
that \(0\) is a root of each \(N_n\), and that the nonzero roots of \(N_{n-1}(q)\) strictly interlace
the nonzero roots of \(N_n(q)\).

By Corollary \ref{properties}, the function
\[
q:(-\infty,0]\to(-\infty,0]
\]
is an order-preserving bijection, and therefore so is \(q^{-1}\).

Since \(r_2(u)\neq 0\) for \(u\le 0\), it follows from (\ref{bona-nar-explicit}) that the roots of \(\mathrm{BoNa}_n(u)\) are precisely
the preimages under \(q\) of the roots of \(N_n\). In particular, \(0\) is a root of \(\mathrm{BoNa}_n(u)\),
and the nonzero roots of \(\mathrm{BoNa}_{n-1}(u)\) are the images under \(q^{-1}\) of the nonzero roots
of \(N_{n-1}(q)\). Since \(q^{-1}\) preserves order, the strict interlacing of the nonzero roots is preserved.

Therefore, the polynomials \(\{\mathrm{BoNa}_n(u)\}_{n\ge2}\) have weakly interlacing roots.
\end{proof}

\section{A three-term recurrence} \label{rec}
In this section, we use a known three-term recurrence for the Narayana polynomials and the relation (\ref{bona-nar-explicit}) between the Narayana
polynomials and the Boolean-Narayana polynomials to prove a three-term recurrence for the latter.

Recall that we defined $r_1(u)$ and $r_2(u)$ right after (\ref{defrs}). An equivalent definition is that they are  the two roots of
\begin{equation}\label{eq:roots}
t^2+(1+u)t+2u=0.
\end{equation}
Also recall that this means that
$r_1(u)+r_2(u)=-(1+u)$, and $ r_1(u)r_2(u)=2u$, and keep the notation
$q(u)=\frac{r_1(u)}{r_2(u)}$.

We will use the standard three-term recurrence for the Narayana polynomials  \cite{kostov}, which states that
\begin{equation}\label{eq:NarayanaRec}
(n+1)N_n(x)=(2n-1)(1+x)N_{n-1}(x)-(n-2)(x-1)^2N_{n-2}(x)
\end{equation}
if $n\geq 3$. 

Substitute $x=q(u)$ in \eqref{eq:NarayanaRec} and multiply both sides by
$u(-1)^{n-1}r_2(u)^{n-1}$. Using \eqref{bona-from-nar} that describes the relation between the polynomials $\bona_n$ and $N_n$, 
 we obtain
\begin{align}
(n+1)\mathrm{BoNa}_n(u)
&=(2n-1)\,u(-1)^{n-1}r_2^{n-1}\bigl(1+q\bigr)\,N_{n-1}(q)
-(n-2)\,u(-1)^{n-1}r_2^{n-1}\bigl(q-1\bigr)^2\,N_{n-2}(q)\nonumber\\
&=-(2n-1)\,r_2(u)\bigl(1+q(u)\bigr)\,\mathrm{BoNa}_{n-1}(u)
-(n-2)\,r_2(u)^2\bigl(q(u)-1\bigr)^2\,\mathrm{BoNa}_{n-2}(u).
\label{eq:BoNaRecIntermediate}
\end{align}

The following identities enable us to eliminate $r_1$, $r_2$, and $q$ from the obtained recurrence.
First,
\[
-r_2(1+q)=-r_2\left(1+\frac{r_1}{r_2}\right)=-(r_1+r_2)=1+u.
\]
Second,
\[
r_2^2(q-1)^2=r_2^2\left(\frac{r_1-r_2}{r_2}\right)^2=(r_1-r_2)^2.
\]
Using \eqref{eq:roots},
\[
(r_1-r_2)^2=(r_1+r_2)^2-4r_1r_2=(1+u)^2-8u=u^2-6u+1.
\]
Substituting these into \eqref{eq:BoNaRecIntermediate} yields the  three-term recurrence that was our goal to prove in this section.
\begin{theorem} \label{3rec}
For all $n\geq 3$, the recurrence relation
\begin{equation}\label{eq:BoNaRecFinal}
\ (n+1)\,\mathrm{BoNa}_n(u)=(2n-1)(1+u)\,\mathrm{BoNa}_{n-1}(u)
-(n-2)(u^2-6u+1)\,\mathrm{BoNa}_{n-2}(u),
\end{equation}
holds. 
\end{theorem}

With initial values 
\[
\mathrm{BoNa}_1(u)=u,\qquad \mathrm{BoNa}_2(u)=u+u^2,
\]
the recurrence \eqref{eq:BoNaRecFinal} determines all polynomials $\mathrm{BoNa}_n(u)$.

\section{Proving the interlacing property from the three-term recurrence} \label{Sturm}
Theorem \ref{3rec} and a result of Liu and Wang \cite{liuwang} yield an alternative proof for Theorem \ref{interlacing} that states that 
the polynomials $\bona_n(u)$ have weakly interlacing roots. The  mentioned result from \cite{liuwang} is more general than what we need here,
and we only cite the case that we will use. 

\begin{lemma} (Corollary 2.4 in \cite{liuwang}) \label{liw}  Let $\{P_n(x)\}_n$ be  a sequence of polynomials with positive coefficients
  that satisfies the recurrence relation
\begin{equation} \label{liuwangrec} P_n(x)=a_n(x)P_{n-1}(x)+b_n(x)P_{n-1}'(x)+c(x)P_{n-2}(x), \end{equation}
where $a_n$, $b_n$ and $c_n$ are polynomials with real coefficients so that 
$\mathrm{deg} P_n = \mathrm{deg} P_{n-1}+1$. If $b_n(x)\leq 0$ and $c_n(x)\leq 0$ whenever $x\leq 0$, 
and at least one of $b_n(x)<0$ and $c_n(x)<0$ holds for all $n$, then the roots of $P_n(x)$ and $P_{n-1}(x)$
are weakly interlacing.
\end{lemma}

Here is the alternative proof of Theorem \ref{interlacing} that we promised. 

\begin{proof} Let $n\geq 3$. 
Divide \eqref{eq:BoNaRecFinal} by $(n+1)$ to obtain
\[
\bona_n(u)=a_n(u)\,\bona_{n-1}(u)+b_n(u)\,\bona'_{n-1}(u)+c_n(u)\,\bona_{n-2}(u),
\]
where
\[
a_n(u)=\frac{2n-1}{n+1}(1+u),\qquad b_n(u)=0,\qquad
c_n(u)=-\frac{n-2}{n+1}(u^2-6u+1).
\]
Since $b_n(u)\equiv 0$, we have $b_n(u)\le 0$ for all $u\le 0$.
Moreover, for $u\le 0$ one has $u^2-6u+1=u^2+(-6u)+1\ge 1$, hence for $n\ge 3$, we have
\[
c_n(u)=-\frac{n-2}{n+1}(u^2-6u+1)<0\qquad\text{for all }u\le 0.
\]
By construction, each $\bona_n(u)$ has nonnegative coefficients, and $\deg \bona_n=n$.
So,  the hypotheses of Lemma \ref{liw} are satisfied; therefore,  each polynomial $\bona_n$ is real-rooted with simple zeros, and the  roots of $\bona_{n-1}$
weakly interlace the roots of $\bona_n$ for all $n\ge 2$.
\end{proof}

\section{Log-concavity for fixed $k$}  \label{logc-vert}
\begin{theorem} For any fixed $k>1$, the sequence $\{\bona(n,k)\}_{n\geq k+1}$ is log-concave.
\end{theorem}

In other words, the Boolean-Narayana numbers are "vertically log-concave". 
We will need the following lemma, which is Proposition 2.5.1 in \cite{brenti}. Note that we say that a sequence of real numbers $z_0,z_1,\cdots $ has
no internal zeros if there are no three indices $i<j<k$ so that $z_iz_k\neq 0$ but $z_j=0$.

\begin{lemma} \label{bintransform} Let $a$ and $b$ be two nonnegative integers, and let $z_0,z_1,\cdots $ be a log-concave sequence of nonnegative real numbers with no internal zeros.
Then the sequence $v_0,v_1,\cdots $ defined by
\[v_n = \sum_{j=0}^n {n \choose j}z_j =  \sum_{j\geq 0} {n \choose j}z_j \]
is log-concave.
\end{lemma}
This result is often mentioned by saying that the {\em binomial transform preserves log-concavity} as the expression on the right-hand  side is the binomial transform
of the sequence $z_n$.

In order to be able to use Lemma \ref{bintransform}, we rearrange the expression obtained for $\bona(n,k)$ in Theorem \ref{explicit} so that it represents the
binomial transform of a log-concave sequence. We achieve this as follows.
First, we note that $\frac{1}{n}{n\choose k-1} = \frac{1}{k-1}{n-1\choose k-2}$. This turns the formula from Theorem \ref{explicit} into
\begin{equation} \label{firststep} \bona(n,k)=\frac{1}{k-1} {n -1 \choose k-2}   \sum_{j=0}^{k-1} 2^j {k-1\choose j}{n-k+1\choose j+1} .\end{equation}
Note that after expanding the product above, we can make good use of the identity
\[{n-1\choose k-2}{n-k+1\choose j+1} = {k+j-1\choose k-2}{n-1\choose k+j-1} .\]
Indeed, it is routine to check that the two sides are identical  by simply using the definition of binomial coefficients. The advantage of the format on the right-hand
side is that only one of the factors there depends on $n$. Comparing the last displayed equation with (\ref{firststep}), we see that
\begin{equation} \label{secondstep} \bona(n,k)=\sum_{j=0}^{k-1} a_j {n-1\choose k+j-1},\end{equation}
where \[a_j=\frac{2^{j}}{k-1} {k-1\choose j} {k+j-1\choose k-2}.\]
Note that the $a_j$ depend on $k$ and $j$ only, and not on $n$.

Let us define a sequence $c_m$ so that $c_m=a_{m-(k-1)}$ if $k-1\leq m\leq 2k-2$, and $c_m=0$ otherwise. Then setting $j=m-(k-1)$, we see that (\ref{secondstep}) becomes
\[\bona(n,k)=\sum_{m\geq 0}c_m {n-1\choose m},\] and our proof is complete by Lemma \ref{bintransform} with $n-1$ playing the role of $n$.
Lemma \ref{bintransform} applies, because the sequence of the $c_m$ has no internal zeros. 
Indeed, \(c_m=0\) if \(m<k-1\) or if \(m>2k-2\). On the other hand, if
\(k-1\le m\le 2k-2\), then \(j=m-(k-1)\) satisfies \(0\le j\le k-1\), so
\[
c_m=a_{m-(k-1)}=a_j.
\]
By the definition of \(a_j\),
\[
a_j=\frac{2^j}{k-1}\binom{k-1}{j}\binom{k+j-1}{k-2},
\]
and each factor on the right-hand side is positive. Therefore \(c_m>0\) for
all \(m\) satisfying \(k-1\le m\le 2k-2\). Hence the sequence \((c_m)\) has
no internal zeros.

\section{Further directions} \label{further}
It is well-known that unlabeled binary plane trees are in natural bijection with 231-avoiding permutations of length $n$. On the other hand, {\em decreasing binary trees}, that is, plane binary trees whose vertices are bijectively labeled with the numbers $1,2,\cdots ,n$ so that each vertex has a smaller label than its parent, are in natural bijection with all permutations of length $n$. Therefore, these trees are counted, respectively, by the Catalan numbers and factorials. In both cases, we can refine our counting
argument to focus on trees with a given number of right edges, and then we get, respectively, the Narayana numbers and the Eulerian numbers.  

In this paper, we have shown that 0-1 trees are an intermediate class of trees between those two mentioned in a previous paragraph, and we provided combinatorial 
interpretations for them in terms of permutations. We hope to extend on this in a future paper by finding other natural classes of trees which are in natural bijections
with permutations so that the number of right edges of the trees corresponds to the number of descents of the permutations.

\begin{center} {\bf Acknowledgement}  \end{center}
I am grateful to Per Alexandersson who pointed me in the direction of the three-term recurrence relation proved in Section \ref{rec}, and its alternative proof.

\end{document}